\documentclass{amsart}
\usepackage{amssymb}
\usepackage{graphicx,overpic} % Required for inserting images
\usepackage{color}

\usepackage[hidelinks,pdftex]{hyperref}

\AtBeginDocument{%
   \def\MR#1{}
}

\numberwithin{equation}{section}
\theoremstyle{plain}
\newtheorem{theorem}[equation]{Theorem}
\newtheorem{corollary}[equation]{Corollary}
\newtheorem{lemma}[equation]{Lemma}

\newtheorem*{namedtheorem}{\theoremname}
\newcommand{\theoremname}{testing}
\newenvironment{named}[1]{\renewcommand{\theoremname}{#1}\begin{namedtheorem}}{\end{namedtheorem}}

\theoremstyle{definition}

\newtheorem{question}[equation]{Question}

\newtheorem{claim}[equation]{Claim}

\usepackage{microtype}

\usepackage[dvipsnames]{xcolor}
\definecolor{MyCyan}{HTML}{00F9DE}
\usepackage{enumitem} 
\setlist{nolistsep,leftmargin=*}
\usepackage{transparent} 
\usepackage{mathrsfs}

\usepackage{marginnote}
\long\def\@savemarbox#1#2{\global\setbox#1\vtop{\hsize\marginparwidth 
  \@parboxrestore\tiny\raggedright #2}}
\renewcommand{\setminus}{{\smallsetminus}}

\DeclareMathOperator{\PSL}{PSL}

\DeclareMathOperator{\Vol}{Vol}

\DeclareMathOperator{\Ker}{Ker}

\newcommand{\fun}[3]{#1 \colon #2 \to #3}

\title{Geometry and topology of closed geodesics complements in the 3-torus}

\author{Jos\'{e} Andr\'{e}s Rodr\'{\i}guez Migueles}
\address[]{Centro de Investigaci\'on en Matem\'aticas, GTO 36023, Mexico}
\email[] {jose.migueles@cimat.mx} 

\date{\today}

\begin{document}

\begin{abstract}
We show that for at most three closed geodesics with linearly independent directions, the homeomorphism type of its complement in the $3$-torus is determine by the orbit of their direction vectors subspaces under the action of $\PSL_3(\mathbb{Z}).$ Moreover, we provide asymptotically sharp volume bounds for a  family of closed geodesics complements. The bounds depend only on the distance in the Farey graph.
\end{abstract}

\maketitle

\section{Introduction}
Let $\mathbb{T}^3$ be the $3$-torus, this three-dimensional manifold admits an Euclidean structure obtained as the orbit space of $\mathbb{R}^3,$ under the discrete group $\mathbb{Z}^3$ acting as integer translations, the quotient map is denoted by $\mathcal{P}$ . Every closed geodesic $R$ is an embedding of $\mathbb{S}^1$ into $\mathbb{T}^3,$ so it can then be considered as a knot in $\mathbb{T}^3.$ 

The aim of this paper, is to study the geometry and topology of the complement of a finite collection a closed geodesics inside $\mathbb{T}^3.$ Hui and Purcell gave in~\cite{HuiPurcell:GeomRod} and~\cite{Hui:GeomClassificationRod}  an explicit characterization of the JSJ decomposition of this link complement, under the name of rod complements, by using Thurston’s geometrization theorem. A particular consequence is that each link complement with three or more linearly independent closed geodesics  is hyperbolic or has a unique hyperbolic piece in its JSJ decomposition. In the case when the link complement is hyperbolic, such metric is unique by the
Mostow’s Rigidity Theorem, meaning that any geometric invariant is a topological invariant.

In a recent paper~\cite{HuiPurcellDo:VolRodPackings}  Do, Hui and Purcell, obtained the first volume bounds for the hyperbolic closed geodesics complement in terms of rod parameters (see~\cite[Theorem~3.2]{HuiPurcellDo:VolRodPackings})  and they showed that these bounds may be loose in general. But in~\cite[Theorem~5.7]{HuiPurcellDo:VolRodPackings} they gave an asymptotically sharp volume bounds for a family of closed geodesics complement, in terms of the lengths of the continued fractions formed from the rod parameters.

In this paper, we give asymptotically sharp volume bounds for a bigger family than in~\cite[Theorem~5.7]{HuiPurcellDo:VolRodPackings} in terms of the Farey distance between pair of slopes coming from the projection of two components to two simple closed geodesics on a once-punctured $2$-torus with a non-complete flat metric..

\begin{theorem}\label{Thm:VolRod}
There exist a universal constant $K_1\geq 1,$ such that for at least three  disjoint closed geodesics $R_1, R_2,...,R_n$ in the $3$-torus where $R_n$ has a direction vector $(0,0,1)$, for $i<n,$ $R_i$ is the image under $\mathcal{P}$ of the geodesic $\{\bar{x}_i+(tp_i,tq_i,s_i)|t\in\mathbb{R}\}\subset\mathbb{R}^3$ with  $0<s_i<s_j<1$ if $i<j,$ $[{p_i}:{q_i}]\neq [{p_{i+1}}:{q_{i+1}}]$, for $i=1,2,...,n-2,$ and $[{p_1}:{q_1}]\neq [{p_{n-1}}:{q_{n-1}}]$.  Then 
{\small{$$ \frac{1}{2K_1}  \sum_{i=1}^{n-1} d_\mathcal{F}\left([{p_i}:{q_i}],[{p_{i+1}}:{q_{i+1}}]\right) \leq\Vol(\mathbb{T}^3\setminus \{R_i\}_{i=1}^n)\leq v_8\sum_{i=1}^{n-1} d_\mathcal{F}\left([{p_i}:{q_i}],[{p_{i+1}}:{q_{i+1}}]\right),$$}}
where $v_8$ is the volume of a regular ideal octahedra and $[{p_n}:{q_n}]:=[{p_1}:{q_1}]$.
\end{theorem}
Although the upper bound is explicit in general, and sharp when the closed geodesis $R_1,...,R_{n-1}$ is a collection of Farey neighbors  (see~\cite[Section~4]{PurcelTaliRodriguez:ArithmeticModularMKnot} and~\cite[Corollary~7.3]{PurcelTaliRodriguez:ArithmeticModularMKnot}) our lower bound depends on a non-explicit universal positive constant. This makes difficult to compare the bounds in~\cite[Theorem~5.7]{HuiPurcellDo:VolRodPackings} with respect Theorem~\ref{Thm:VolRod}. However we can relate the invariant of the lengths of the continued fractions of the slope of the closed geodesic with respect the Farey distance between infinity and the corresponding slope, because they are the same. For that reason, we can conclude by triangle inequality that our upper bound is sharper in general and equal for some special cases.

Furthermore, we also investigate the question of the characterization on the topological type of the closed geodesics complement in the following new result which is the three-dimensional and at most three linearly independent component version of~\cite[Proposition~2.5]{HuiPurcellDo:VolRodPackings}:

\begin{theorem}\label{Thm:MainHomeotTypeClassification3Rod}
Let $\{R_i\}_{i=1}^k$ and $\{R'_i\}_{i=1}^k$ be a pair of $k\leq 3$  disjoint linearly independent closed geodesics in the 3-torus. Then $$\mathbb{T}^3\setminus \{R_i\}_{i=1}^k\cong \mathbb{T}^3\setminus \{R'_i\}_{i=1}^k,$$ if and only if there exist $A\in \PSL_3(\mathbb{Z})$ such that maps the subspace generated by the direction vector of $R_i$ to the subspace generated by $R'_i$ for all $i.$
\end{theorem}

\subsection{Acknowledgements} I was  greatly benefited with conversations from  Araceli Guzmán-Tristán, Connie On Hui, Jessica Purcell and Jesús Rodríguez-Viorato.
I also thank CIMAT and my first semester students in the Geometry class for creating an attractive mathematical environment.
\section{Preliminaries}
Let $\mathbb{T}^3$ be the $3$-torus obtained by gluing identically the opposite faces
of the unit cube $[0, 1]\times[0, 1] \times [0, 1]$ in the 3-dimensional Euclidean space. Equivalently, we can see  $\mathbb{T}^3$ as the orbit space of $\mathbb{R}^3,$ under the discrete group $\mathbb{Z}^3$ acting as integer translations. This quotient map $\fun{\mathcal{P}}{\mathbb{R}^3}{\mathbb{T}^3}$ is also a Universal Riemannian covering map (see Definition 2.1 in~\cite{HuiPurcell:GeomRod} for an explicit map). Then $\mathbb{T}^3$  inherits the Euclidean metric from $\mathbb{R}^3.$ 

In this section, we fix a parametrization for closed geodesics and totally geodesic $2$-torus inside the $3$-torus, following  ideas in~\cite{HuiPurcell:GeomRod} and~\cite{Hui:GeomClassificationRod}, but instead of calling our links rods we called them closed geodesics.
\subsection{Closed geodesics in the $3$-torus}
Recall that a geodesic in the three-dimensional Euclidean space can be given by a parametric expression of the form:

$$\{\bar{p}+t\vec{d}\,|\,t\in\mathbb{R}\}$$

where $\bar{p}\in\mathbb{R}^3$ and $\vec{d}\in\mathbb{R}^3\setminus\{(0,0,0)\}$ is the direction vector of the geodesic.

Notice that this parametrization is not done by arc-length. And in many cases does not projects injectivelly under the universal covering map to $\mathbb{T}^3.$ But as we will be only concern about the closed geodesic as a knot, we do not inquire in a precise injective arc-length parametrization.

Recall a version of Lemma 2.2 in~\cite{HuiPurcell:GeomRod} that  characterizes a closed geodesic in $\mathbb{T}^3,$ without precising the parametrization of the geodesic.

\begin{lemma}\label{Lem:ClosedGeodesic3Torus}
A geodesic in three-dimensional Euclidean space projects under $\mathcal{P}$ to a closed geodesic in $\mathbb{T}^3,$ if and only if,  the coordinates of its direction vector generate a one dimensional vector space over $\mathbb{Q}$.
\end{lemma}

\subsection{Totally geodesic $2$-torus inside the $3$-torus}
Recall that a Euclidean plane in the three-dimensional Euclidean space can be given by a parametric expression of the form:

$$\{\bar{p}+t\vec{d_1}+s\vec{d_2}\,|\,t,s\in\mathbb{R}\}$$

where $\bar{p}\in\mathbb{R}^3$ and $\vec{d_1},\vec{d_2}\in\mathbb{R}^3\setminus\{(0,0,0)\}$ are linearly independent direction vectors of the plane.

Recall the Lemma 2.6 in~\cite{Hui:GeomClassificationRod}  characterizes a totally geodesic $2$-torus in $\mathbb{T}^3.$

\begin{lemma}\label{Lem:TorusGeodesic3Torus}
A plane in three-dimensional Euclidean space projects under $\mathcal{P}$ to a totally geodesic $2$-torus in $\mathbb{T}^3,$ if and only if,  the plane has two linearly independent director vectors such that the coordinates of each one generate a one dimensional vector space over $\mathbb{Q}$.
\end{lemma}

Notice that the difference in the statement of Lemma 2.6 in~\cite{Hui:GeomClassificationRod} rely on the flexibility of the way we do not precise the parametrization of the plane.
\section{Homeomorphism type of three closed geodesics complement}

In this section we characterize the  homomorphism type of the complement of at most three closed geodesics that are linearly independent in terms of the action of $\PSL_3(\mathbb{Z})$ in at most three non-colinear points in the real projective plane  $\mathbb{P}^2$ (or the subspaces generated by the corresponding vector directions).

\begin{named}{Theorem~\ref{Thm:MainHomeotTypeClassification3Rod}}
Let $\{R_i\}_{i=1}^k$ and $\{R'_i\}_{i=1}^k$ be a pair of $k\leq 3$  disjoint linearly independent closed geodesics in the 3-torus. Then $$\mathbb{T}^3\setminus \{R_i\}_{i=1}^k\cong \mathbb{T}^3\setminus \{R'_i\}_{i=1}^k,$$ if and only if there exist $A\in \PSL_3(\mathbb{Z})$ such that maps the subspace generated by the direction vector of $R_i$ to the subspace generated by $R'_i$ for all $i.$
\end{named}
\begin{proof}
On one side, suppose there  exist $A\in \PSL_3(\mathbb{Z})$ such that sends the direction vector of $R_i$ to $R'_i$ for all $i.$ Because $A\in \PSL_3(\mathbb{Z})$ then the linear map associated in $\mathbb{R}^3$ passes to an homeomorphism $\bar{A}$ of the quotient space $\mathbb{T}^3.$ 

If $\bar{A}(R_1)\neq R'_1,$ then there is a  translation $\tau$ in $\mathbb{T}^3$ that make them coincide, and we have the case  $\mathbb{T}^3\setminus R_1\cong \mathbb{T}^3\setminus R'_1$ . In the case $k=2,$ if we  have that $\tau \circ \bar{A}(R_2)$ and $ R'_2,$ are linearly isotopic in $\mathbb{T}^3\setminus R'_1,$ so $\mathbb{T}^3\setminus \{R_i\}_{i=1}^2\cong \mathbb{T}^3\setminus \{R'_i\}_{i=1}^2.$

For the case $k=3$, if $\tau\circ \bar{A}(R_i)$ is not linear isotopic to $R'_i$ for $i=1,2$ in the complement of $\tau\circ \bar{A}(R_j)\cup R'_j$ with $i\neq j$ and $j=2,3,$ then there is a translation $\tau'$ in $\mathbb{T}^3$  with the same vector direction  as $\tau\circ \bar{A}(R_1)$  such that makes them linear isotopic in the corresponding complements.

Finally, up to disjoint local linear isotopies between  $\tau'\circ\tau\circ \bar{A}(R_i)$ and $R'_j$ for $i=1,2$ we found an homeomorphism of $\mathbb{T}^3$ sending  $(R_1, R_2,R_3)$ to $(R'_1, R'_2,R'_3),$ giving the wanted homeomorphism between the corresponding complements. 

On the other side, is enough to show that the homeomorphism $h$ between $\mathbb{T}^3\setminus\{R_i\}_{i=1}^k$ and $\mathbb{T}^3\setminus \{R'_i\}_{i=1}^k$ can be extended to a self-homeomorphism $\widehat h$ of $\mathbb{T}^3$ with $\widehat h(R_i)=R'_i.$ Notice that if $h$ maps $m_{i},$ the meridian of $R_i,$ to the meridian of $R'_i,$ then we can extend $h$ to the disks that they bound inside $\mathbb{T}^3$. Moreover, we can extend this homeomorphism along the core of $\mathcal{N}_{R_i},$ the normal neigborhood of $R_i,$ to $\mathcal{N}_{R'_i}.$ As the morphism in homology induced by the restriction of $h$ on $\partial\mathcal{N}_{R_i}$ is an isomorphism, then $H_1(h_{\mid\partial \mathcal{N}_{R_i}})(\Ker(H_1(i_{R_i})))=\Ker(H_1(i_{R'_i})),$ where $i_{R_i}$ is the inclusion of $\partial\mathcal{N}_{R_i}$ into $\mathbb{T}^3\setminus \{R_i\}_{i=1}^k,$ and analogous for $i_{R'_i}.$ As the image of a essential simple closed curve that represents the class of the meridian under the homeomorphism $h_{\mid\partial \mathcal{N}_{R_i}}$ has to be a essential simple closed curve inside $\partial \mathcal{N}_{R'_i},$ then we just need to prove the following claim\string:

\begin{claim}\label{Claim:Meridian}
$\Ker(H_1(i_{R_i}))$ is generated by a multiple of  $[m_{R_i}]$ in $H_1(\partial \mathcal{N}_{R_i}; \mathbb{R}).$
\end{claim}
To proof the Claim \ref{Claim:Meridian} notice that the image of the homology class of a longitude in $\partial \mathcal{N}_{R_i}$ under $H_1(i_{R_i})$ is not trivial in $H_1(\mathbb{T}^3\setminus\{R_i\}_{i=1}^k; \mathbb{R})$ because $[R_i]$ is not trivial in $H_1(\mathbb{T}^3; \mathbb{R}).$ Then by considering the Mayer-Vietoris sequence in homology for the triad $(\mathbb{T}^3, \mathbb{T}^3\setminus\{R_i\}_{i=1}^k,\bigcup_{i=1}^k \mathcal{N}_{R_i}),$ we only need to show that there is an element in $H_2(\mathbb{T}^3)$ whose image  under $\widehat\delta$, the connecting morphism of the sequence, is a non trivial element of $ H_1(\partial \mathcal{N}_{R_i}).$
\vskip .2cm
Consider $T_i$ the torus in $\mathbb{T}^3$ parallel to the plane that contains the direction vectors of the $R_k$'s but not the vector direction of $R_i$.  Then $$\widehat\delta([T_i])=[\partial(T_i\cap \mathcal{N}_{R_i})]=n[m_{R_i}],\hspace{.2cm}\mbox{with} \hspace{.2cm}n\neq 0.$$
\end{proof}
Now we give an application of this result to a special family of links coming from periodic orbits of the geodesic flow on surfaces. We recall that the projective tangent bundle over $\mathbb{T}_*^2 $ a once-punctured torus with a flat metric, denoted by $PT(\mathbb{T}_*^2),$ is homeomorphic to the complement of a closed geodesic in the $3$-torus, which is the oriented trivial circle bundle over $\mathbb{T}_*^2.$ To every closed geodesics $\gamma$ on $\mathbb{T}_*^2,$  we associate the complement in $PT(\mathbb{T}_*^2 )$ of the periodic orbit of the geodesic flow associated to $\gamma$ which we denoted by $\widehat\gamma.$ Then we obtained a criterium to decide when a couple of two periodic orbits of the geodesic flow have not homeomorphic complement, which is an analog version of Theorem 1.6 in \cite{Rodriguez:OldLowerbound}, for a two link component case.

\begin{corollary}\label{Cor:HomeotTypePeriodicOrbitComplement}
     Let $\{\gamma_1,\gamma_2\}$ and $\{\eta_1,\eta_2\}$ a couple of two non-homotopic simple closed geodesics in $\mathbb{T}_*^2$ relative to a metric with nowhere positive curvature. Then 
    there is an homeomorphism, preserving the cusp coming from the puncture, between $PT(\mathbb{T}_*^2)\setminus\{\widehat\gamma_1,\widehat\gamma_2\}$ and  $ PT(\mathbb{T}_*^2)\setminus\{\widehat\eta_1,\widehat\eta_2\}$ if and only if there is a diﬀeomorphism $\phi$ on $\mathbb{T}_*^2$ such that $\phi(\gamma_i)$ is homotopic to $\eta_i$ for $i=1,2.$
\end{corollary}
\begin{proof}
    Suppose there exist a diﬀeomorphism $\phi$ on $\mathbb{T}_*^2$ such that $\phi(\gamma_i)$ is homotopic to $\eta_i$ for $i=1,2.$ Then we can adapt a flat metric on $\mathbb{T}_*^2$ such that the original simple closed geodesics are transversaly homotopic to the the ones in the flat metric. Notice that the linear isotopy class is determine by their slopes $[\gamma_i]$ and $[\eta_i],$ and that a  transveral homotopy lifts to an isotopy of the corresponding periodic orbits (see  \cite[Subsection 2.1]{Rodriguez:OldLowerbound}). 
    
    Moreover, the existence of $\phi$ induces an element $A\in\PSL_2(\mathbb{Z})$ such that $A[\gamma_i]=[\eta_i].$ Then naturally we can extend $A$ to $\bar{A}\in\PSL_3(\mathbb{Z})$ which fixes the last coordinate, and this induces an homeomorphism $\widehat{\phi}$ on $\mathbb{T}^3$ which fixes the closed geodesic with direction vector $(0,0,1).$ Up to composing with a translation parallel  to the direction $(0,0,1)$ we get that $\widehat{\phi}(\widehat\gamma_i)$ and $\widehat\eta_i$ are linearly isotopic in the complement and leaves invariant the cusp of $PT(\mathbb{T}_*^2)$ coming from the puncture of $\mathbb{T}_*^2$ . This gives the wanted homeomorphism between the corresponding complements.

    If we suppose that there is an homeomorphism $$\fun{h}{PT(\mathbb{T}_*^2)\setminus\{\widehat\gamma_1,\widehat\gamma_2\}}{PT(\mathbb{T}_*^2)\setminus\{\widehat\eta_1,\widehat\eta_2\}},$$ preserving the cusp coming from the puncture,  then by Theorem~\ref{Thm:MainHomeotTypeClassification3Rod} this extends to an homeomorphism of the $3$-torus which, up to isotopy, lifts to a linear map $\bar{A}\in\PSL_3(\mathbb{Z})$ which fixes the last coordinate because it preserves cusp coming from the fibers in $PT(\mathbb{T}_*^2).$ Then by forgetting the last coordinate on the linear map  $\bar{A}$ induces a linear $A\in\PSL_2(\mathbb{Z}),$ which induces an diffeomorphism of $\mathbb{T}_*^2$ sending $\gamma_i$ to a simple closed geodesic linearly isotopic to $\eta_i.$
\end{proof}

\section{Volume bounds for stratified closed geodesics complements}

\subsection{Farey graph}
Let $\mathbb{T}^2$ be the $2$-torus,  obtained by gluing identically the opposite sides of the unit square $[0, 1]\times[0, 1]$ in the 2-dimensional Euclidean plane,  and $\mathbb{T}_*^2$ the once-punctured $2$-torus with a non-complete flat metric. The Farey graph $\mathcal{F}$  is the metric graph obtained by taking the set of isotopy classes of simple closed geodesics as vertices (which is in correspondence with the rational points of the projective line $\mathbb{P}^1).$ We say that $[{p}:{q}],[{r}:{s}]\in \mathbb{P}^1$ two vertices are connected by an edge if the corresponding closed geodesics intersect only once (this does not depend on the representant of the class). For more facts about the Farey graph see Subsection 2.3 in~\cite{PurcelTaliRodriguez:ArithmeticModularMKnot}.

\begin{named}{Theorem~\ref{Thm:VolRod}}
There exist a universal constant $K_1\geq 1,$ such that for at least three  disjoint closed geodesics $R_1, R_2,...,R_n$ in the $3$-torus where $R_n$ has a direction vector $(0,0,1)$, for $i<n,$ $R_i$ is the image under $\mathcal{P}$ of the geodesic $\{\bar{x}_i+(tp_i,tq_i,s_i)|t\in\mathbb{R}\}\subset\mathbb{R}^3$ with  $0<s_i<s_j<1$ if $i<j,$ $[{p_i}:{q_i}]\neq [{p_{i+1}}:{q_{i+1}}]$, for $i=1,2,...,n-2,$ and $[{p_1}:{q_1}]\neq [{p_{n-1}}:{q_{n-1}}]$.  Then 
{\small{$$ \frac{1}{2K_1}  \sum_{i=1}^{n-1} d_\mathcal{F}\left([{p_i}:{q_i}],[{p_{i+1}}:{q_{i+1}}]\right) \leq\Vol(\mathbb{T}^3\setminus \{R_i\}_{i=1}^n)\leq v_8\sum_{i=1}^{n-1} d_\mathcal{F}\left([{p_i}:{q_i}],[{p_{i+1}}:{q_{i+1}}]\right),$$}}
where $v_8$ is the volume of a regular ideal octahedra and $[{p_n}:{q_n}]:=[{p_1}:{q_1}]$.
\end{named}
\begin{proof}
    The volume upper bound is obtained by removing in between the $ R_i$ and the $ R_{i+1} $ closed geodesic a finite ordered sequence of perpendicular closed geodesics to the vector direction $(0,0,1)$ whose corresponding vector direction corresponds to the vertices of the geodesic in Farey graph joining $[{p_i}:{q_i}]$ and $[{p_{i+1}}:{q_{i+1}}]$. Since Dehn filling does not increase volume by Thurston \cite[Proposition~6.5.2]{Thurston:Geom&TopOf3Mfd}, we have an upper bound  by estimating the volume of this new drilled manifold. By~\cite[Theorem~4.2]{PurcelTaliRodriguez:ArithmeticModularMKnot} the resulting  manifold has a complete hyperbolic structure obtained by gluing $\sum_{i=1}^{n-1} d_\mathcal{F}\left([{p_i}:{q_i}],[{p_{i+1}}:{q_{i+1}}]\right)$ regular ideal octahedra.

    For the volume lower bound, we used the fact that the image under $\mathcal{P}$ of $\{(R_i, z=s_i)\}_{i=1}^{n-1}$ minus the puncture coming from $R_n$ is a stratification of $R_i$ with $i=1,...,n-1$ inside $\mathbb{T}^3\setminus  R_n$ (see~\cite[Definition~1]{CRMY} ) and~\cite[Theorem~E]{CRMY} there is a positive constant $K_1\geq 1$ depending only on the topology of the once-punctured torus, such that:
    $$\frac{1}{2K_1} \left(\sum_{i=1}^{n-1} d_\mathcal{F}(P_{X_i},P_{Y_{i+1}}) \right)\leq Vol(\mathbb{T}^3\setminus \{R_i\}_{i=1}^n)$$
   where $\{   (P_{X_i},P_{Y_{i+1}}) \}_{i=1}^{n-1}$ are the Bers pants decomposition coming from the conformal structures of the stratifying surfaces, which is the projection of $z=s_i,$ minus the point $(0,0,s_i),$ under $\mathcal{P}$ . 
   
   Since any such pants decomposition contains only one loop the pants distance is the same as the Farey graph distance. Moreover, this loop is forced to be the rank one cusp induced by a simple closed curve of slope $[{p_i}:{q_i}]$. Therefore, we see that:
$$  d_\mathcal{F}(P_{X_i},P_{Y_{i+1}})= d_\mathcal{F}\left([{p_i}:{q_i}],[{p_{i+1}}:{q_{i+1}}]\right)$$
   
\end{proof}
\begin{corollary}\label{Cor:Vol3Rod}
There exist universal constant $K_1\geq 1,$ such that for $R_1, R_2,R_3$ linear independent closed geodesics such that de direction vector of $R_3$ is orthogonal to the plane generated by the direction vectors of $R_2$ and $R_1$. Then 
$$\frac{1}{2K_1}  d_\mathcal{F}\left([{p_1}:{q_1}],[{p_{2}}:{q_{2}}]\right) \leq \Vol(\mathbb{T}^3\setminus \{R_i\}_{i=1}^3)\leq 2 v_8 d_\mathcal{F}\left([{p_1}:{q_1}],[{p_{2}}:{q_{2}}]\right),$$

where $[{p_i}:{q_i}]$ is the direction vector of $R_i$, form $i=1,2$ relative to an orthogonal basis of the plane generated by the direction vectors of $R_2$ and $R_1$ and $v_8$ is the volume of a regular ideal octahedra.
\end{corollary}
\begin{proof}
    By Theorem~\ref{Thm:MainHomeotTypeClassification3Rod} we can apply an an ortogonal matrix in $\PSL_3(\mathbb{Z})$ such that sends the direction vector subespaces of $R_3$ to $(0,0,1)$ and the other direction vectors subespaces corresponding to $R_2$ and $R_1$ to two direction vectors with last coordinate equal to zero. 
    
    Then after choosing an ortogonal basis for the plane generated by the direction vectors corresponding to $R_2$ and $R_1,$ we obtain the slope of each direction vector of  $R_i$ as $[{p_i}:{q_i}]$. 
    
    Finally, by Theorem~\ref{Thm:VolRod} we get the wanted result for the case with three closed geodesics.
\end{proof}

As in the previous section, we give an application to a special family of links coming from periodic orbits of the geodesic flow on the once-punctured torus $\mathbb{T}_*^2$ with a flat metric. Notice that a closed geodesic $\gamma_i$ in this case is given by the parametrization $\{(x_i,y_i)+t(p_i,q_i)\,|\,t\in\mathbb{R}\},$ then the periodic orbit $\widehat\gamma$ is given by the parametrization $\{(x_i,y_i,\arctan\left(\frac{p_i}{q_i}\right) )+t(p_i,q_i,0)\,|\,t\in\mathbb{R}\}.$ Consequently, given $\Gamma$ any finite family of nonparallel closed geodesics on the once-punctured torus $\mathbb{T}_*^2$ with a flat metric, induces an ordering by calculating $\frac{-\pi}{2}\leq\arctan\left(\frac{p_i}{q_i}\right)<\frac{\pi}{2},$ and then we explicit and improve the right inequality of Theorem B in \cite{CRMY} as follows:

\begin{corollary}\label{Cor:VolPeriodicOrbits}
Let $\Gamma$ is a filling collection of nonparallel essential simple closed curves in minimal position on $\mathbb{T}_*^2.$ Then, there exists an ordering in $\Gamma=\{\gamma_i\}^n_{i=1}$ and a universal constant  $K_1\geq 1,$
such that:

{\footnotesize $$\frac{1}{2K_1} \left(\sum_{i=1}^{n} d_\mathcal{F}\left([{p_i}:{q_i}],[{p_{i+1}}:{q_{i+1}}]\right) \right)\leq \Vol\left(PT(\mathbb{T}_*^2)\setminus\widehat\Gamma\right)\leq 2 v_8 \left(\sum_{i=1}^{n} d_\mathcal{F}\left([{p_i}:{q_i}],[{p_{i+1}}:{q_{i+1}}]\right) \right),$$}

where  $[{p_i}:{q_i}]$ is the slope associated to $\gamma_i,$ $v_8$ is the volume of a regular ideal octahedra and $\gamma_{n+1}:=\gamma_1$.
\end{corollary}

\section{Further questions}
Our results on the geometry and topology of closed geodesics complements suggest various questions worthy of further exploration, such as the following.
\begin{question}
    Does there exist a more general characterization of the topological type of closed geodesics complement in $\mathbb{T}^3$?
\end{question}
A problem appears, for example when there are two  closed geodesics whose directions vectors are linearly dependent and there is another closed geodesic intersecting the annuli between them, because one can obtain an homeomorphism coming by doing annular Dehn filling (see Section 4 in~\cite{HuiPurcellDo:VolRodPackings}). Other possible direction of generalization is for the complements of closed geodesics in $\mathbb{T}^n$ as is stated for the one component case in~\cite[Proposition~2.5]{HuiPurcellDo:VolRodPackings}
\begin{question}
    Do there exist explicit sharp volume bounds for three linear independent closed geodesics complements, such that any pair is orthogonal in between, in terms of same combinatorial invariant of the closed geodesics?
\end{question}

Even though Do, Hui and Purcell, gave in~\cite{HuiPurcellDo:VolRodPackings} explicit volume bounds for all closed geodesics complement, these are loose in general and are not in terms of the same combinatorial invariant of the closed geodesics.

\begin{question}
    For the lower bound in~\cite[Theorem~5.7]{HuiPurcellDo:VolRodPackings} and our Theorem~\ref{Thm:VolRod}, which is sharper?
\end{question}

\bibliography{references}
\bibliographystyle{amsplain}
\end{document}